\documentclass[a4paper,12pt]{article}
\usepackage[margin=3cm]{geometry}
\usepackage{amsmath}
\usepackage{amsthm}
\usepackage{amssymb}
\usepackage{stackrel}
\usepackage{graphicx}
\usepackage{color}

\def\Z{\mathbb Z}
\def\R{\mathbb R}
\def\N{\mathbb N}

\def\B{\mathcal B}
\def\A{\mathcal A}

\def\L{\mathcal L}
\def\pf{\begin{proof}}
\def\pfk{\end{proof}}
\newcommand{\vpb}{\varphi_\beta}
\newcommand{\ub}{\mathbf{u}}
\newcommand{\ubeta}{\ub_\beta}
\newcommand{\Bub}{\B(\ub)}
\newcommand{\Bubeta}{\B(\ubeta)}
\newcommand{\Le}[1]{\mathrm{Lext}(#1)}
\newcommand{\Rex}[1]{\mathrm{Rext}(#1)}
\newcommand{\E}{\mathrm{E}}
\newcommand{\ind}[1]{\text{ind}(#1)}
\newcommand{\e}{\epsilon}
\newcommand{\lcp}[2]{\mathrm{lcp}(#1,#2)}

\newtheorem{pro}{Proposition}
\newtheorem{lem}[pro]{Lemma}

\newtheorem{thm}[pro]{Theorem}
\newtheorem{dfn}[pro]{Definition}
\newtheorem{cor}[pro]{Corollary}

\newtheorem{exa}[pro]{Example}

\title{Critical exponent of infinite words coding beta-integers associated with non-simple Parry numbers}

\author{L{\!'}. Balkov\'a,  K. Klouda, E. Pelantov\'a
\footnote{e-mail: l.balkova@centrum.cz, karel.klouda@fjfi.cvut.cz, edita.pelantova@fjfi.cvut.cz}\\
\emph{Department of Mathematics, FNSPE,
Czech Technical University},\\ \emph{
Trojanova 13, 120 00 Praha 2, Czech Republic}\\}
\begin{document}
\maketitle
\begin{abstract}
In this paper, we study the critical exponent of infinite words $\ubeta$ coding $\beta$-integers for $\beta$ being a~non-simple Parry number.
In other words, we investigate the maximal consecutive repetitions of factors that occur in the infinite word in question.
We calculate also the ultimate critical exponent that
expresses how long repetitions occur in the infinite word $\ubeta$
when the factors of length growing ad infinitum are considered.
The basic ingredients of our method are the
description of all bispecial factors of $\ubeta$ and the notion of
return words. This method can be applied to any fixed point of any
primitive substitution.
\end{abstract}
\section{Introduction}

In this paper the infinite words associated with non-simple Parry
numbers $\beta$ are studied. These words, denoted by $\ubeta$,
have two equivalent definitions, they are the words coding the
gaps between consecutive $\beta$-integers and, at the same time,
they are fixed points of the substitutions $\vpb$ canonically
assigned to $\beta$. Our aim is to find the maximal repetitions of
motifs occurring in $\ubeta$, more precisely, to compute the
critical exponent and the ultimate critical exponent of these words (for definition see
\eqref{criticalexponent} and \eqref{ultimatecriticalexponent}).

The $\beta$-integers proved to be a convenient discrete set for
description of positions of atoms in the materials with long range
order, so-called quasicrystals~\cite{Burdik1998}. Physical
properties of these materials are determined by the spectrum of
the discrete Schr\"{o}dinger operator assigned to this aperiodic structure.
Damanik shown in~\cite{Damanik2000} that there is a strong
connection between the properties of the spectrum and the value of
the critical exponent of the word $\ubeta$.

In 1912, A. Thue studied words with minimal repetitions; he
discovered a word with the critical exponent equal to two -- the
lowest possible critical exponent for binary words -- which is now
known as Thue-Morse word~\cite{Thue1912}. A great effort was made
to compute the critical exponent of Sturmian words. For the most
prominent Sturmian word, namely the Fibonacci word, the critical
exponent was calculated by Mignosi and Pirillo in 1992
in~\cite{Mignosi1992}. The general result for all Sturmian words
was provided independently by Carpi and de Luca~\cite{Carpi2000}
and by Damanik and Lenz~\cite{Damanik2002}; the formula comprises
the coefficients of the continued fraction of the slope of a given
Sturmian word.

The major contribution to the problem of computing the critical
exponent of fixed points of substitution is due to D.
Krieger~\cite{Krieger2007}. She proved that the critical exponent
of such words is either infinite or belongs to the algebraic field
generated by the eigenvalues of the incidence matrix of the
respective substitution.

In the present paper we provide the formula for computing the
critical exponent of the words $\ubeta$ associated to non-simple
Parry numbers. The basic ingredients of our method are the
description of all bispecial factors of $\ubeta$ and the notion of
return words. This method can be applied to any fixed point of any
primitive substitution.

The paper is organized as follows.  In Section \ref{prelim} we
recall basic notions of combinatorics on words and we introduce a
connection of the studied words $\ubeta$ with some numeration
systems. Section \ref{MaxAndBispecial} shows that description of
bispecial factors and return words is crucial  for evaluation of
the critical exponent of any infinite word. Therefore, Section
\ref{bispecialUNAS} is focused on these objects in the word
$\ubeta$. In Section \ref{main}
 the main theorem is stated. Its proof is contained
in Section \ref{proof}. The last section is devoted to derivation
of a simple form of the ultimate critical exponent.

\section{Preliminaries}\label{prelim}

\subsection{Combinatorics on words}

A {\em finite word} $w$ over a finite alphabet $\A = \{a_1, a_2, \ldots,
a_{m}\}$ is a string of letters from $\A$, i.e., $w= w_1w_2\ldots
w_n$, where $w_i \in \A$ for all $i=1,2, \ldots ,n$. The {\em length} of
the word $w=w_1w_2\ldots w_n$ will be denoted $|w| = n$, by
$|w|_a$ we denote the number of occurrences of the letter $a$ in
$w$. The \emph{Parikh vector} of a finite word $w$ is the row
vector $\Psi(w)= (|w|_{a_1}, |w|_{a_2},\ldots, |w|_{a_m}) \in
\N^m$. Clearly, $|w| = \Psi(w)\vec{e}$, where $\vec{e}$ is the
column vector from $\R^m$ whose all entries are equal to 1.

For the set of finite words over the alphabet $\A$, the notation
$\A^*$ is used. An {\em infinite word} $\ub$ over the alphabet $\A$ is a
sequence $\ub = (u_n)_{n\in \mathbb{N}}$ with $u_n \in \A$ for all
$n \in \mathbb{N}$. The set of such sequences is denoted
$\A^\mathbb{N}$. The set $\A^*$ together with concatenation forms
a monoid, with the empty word $\e$ as its neutral element. The
notation $w^k$ for $w \in \A^*$ and $k \in \N$ stands for
concatenation of $k$ copies of the word $w$; the symbol $w^\omega$
means the infinite repetition of $w$.

If a word $w$ arises by concatenation of $x$  and $y$, i.e., $w =
xy$, then $x$ is called a \emph{prefix} of $w$
 and $y$ is a \emph{suffix} of $w$.  The prefix $x$ can be obtained from
 $w$ by the ``inverse" procedure to concatenation, namely by erasing
 the  suffix $y$,  therefore we will also use $x = w
 y^{-1}$ and, analogously, $y =
 x^{-1}w$.  The \emph{cyclic shift} on $\A^*$
 is the mapping $w\to S(w) = awa^{-1}$,
  where $a$ is the last letter of $w$. Any iteration $S^k(w)$  of the cyclic
  shift  for $k\in \mathbb{N}$ is called a \emph{conjugate} of $w$.
We say that a word $w \in \A^*$ is {\em primitive} if it  has $|w|$
conjugates.

A word $w \in \A^*$ is said to be a \emph{factor} of an infinite
word $\ub = (u_n)_{n\in \mathbb{N}}$ if there exists an index $i
\in  \N$ such that $w$ is a prefix of $u_iu_{i+1} \cdots$. The
index $i$ is an \emph{occurrence} of $w$ in  $\ub$. The set of all
factors of $\ub$ is denoted $\L(\ub)$.

An infinite word $\ub$ is \emph{recurrent}, if any of its factors
has at least two occurrences in $\ub$. If, moreover, the gaps
between neighboring occurrences of a given factor are bounded for
any factor, $\ub$ is \emph{uniformly recurrent}.

A word $v$ is called a \emph{power} of $w$ if $v$ is a prefix of
$w^\omega$. If $v$ is not a power of any word $w'$ shorter than
$w$, then  $w$ is the \emph{root} of $v$. The \emph{index} of a
finite word $w \neq \epsilon$ in an infinite word $\ub$ is
$$
\ind{w} = \sup\left\{ \tfrac{|v|}{|w|}  {\mid}  v\in \L(\ub) \ \
\text{and} \  v \text{ is a power of } w\right\}\,.
$$
Let us limit our considerations to uniformly recurrent infinite words.
Under this assumption, any factor $w \in
 \L(\ub)$ has a finite index.
 A power $v$ of $w$
 for which the supremum is attained is called the \emph{maximal power} of $w$
 in $\ub$. The \emph{critical exponent} of an infinite word  $\ub$ is
 defined as
\begin{equation}\label{criticalexponent}
\E( \ub) = \sup \{ \ind{w} \mid  w  \in
 \L(\ub)\}\,.
\end{equation}
In~\cite{Berthe2006}, the authors introduce $\E^{*}(\ub)$ which is
closely related to $\E(\ub)$. The characteristics $\E^{*}(\ub)$
expresses how long repetitions occur in the infinite word $\ub$
when the factors of length growing ad infinitum are considered. In
order to provide an exact definition of $\E^*(\ub)$, let us denote
by $\text{ind}_n(\ub)=\max\{\ind{w} \bigm | w \in {\mathcal
L}(\ub), \ |w|=n\}$. The \emph{ultimate critical exponent} of an
infinite word  $\ub$ is defined as
\begin{equation}\label{ultimatecriticalexponent}
\E^*( \ub) = \limsup_{n \to \infty} \text{ind}_n(\ub).
\end{equation}
Clearly, $\E(\ub)\geq \E^{*}(\ub)$. In case $\E(\ub) \not \in \mathbb Q$, then $\E^{*}(\ub)=\E(\ub)$.

Let us recall that a \emph{morphism} on  $\A^*$ is a mapping
$\varphi: \A^* \to \A^*$ such that
$\varphi(wv)=\varphi(w)\varphi(v)$ for all $w,v \in \A^*$\,. To
any morphism  $\varphi$, one can assign its \emph{incidence matrix}
$M_{\varphi}$ by the prescription $(M_\varphi)_{a,b} =
|\varphi(a)|_b$. The incidence matrix enables to express the
Parikh vector of the image of $w$ by $\varphi$. One has
\begin{equation}\label{parikh}
\Psi(\varphi(w)) =  \Psi(w) M_\varphi\,.
\end{equation}
We say that a substitution morphism  $\varphi$ is \emph{primitive}
if there exists an exponent  $k\in \mathbb{N}$ such that all
entries of $M_\varphi^k$ are positive.

 The image of an infinite word
$\ub$ by $\varphi$ is naturally defined  as $\varphi (\ub) =
\varphi (u_0u_1u_2\ldots ) = \varphi(u_0)\varphi(u_1) \varphi(u_0)
\ldots $. The word $\ub\in \A^\N$ is a \emph{fixed point} of a
morphism $\varphi$ if  $\varphi (\ub) = \ub$. If $\varphi(b)\neq
\epsilon$ for every letter $b \in \A$ and if there exists a letter
$a\in \A$ and $w \in \A^* -\{\epsilon\}$ such that $\varphi(a) =
aw$, then $\varphi$ is called a \emph{substitution}. Any
substitution has at least one fixed point, namely
$\lim\limits_{n\to \infty} \varphi^n(a)$ (taken in the product
topology). A substitution $\varphi$ in general may have more fixed
points. If $\varphi$ is primitive, then  any of its fixed points
is uniformly recurrent and the languages of all fixed points of
$\varphi$ coincide.

Variability in an infinite word $\ub$ is measured by the
\emph{factor complexity}  $\mathcal{C}: \N \to \N$ defined for
every $n \in \N$ by
$$
\mathcal{C}(n) = \#\{ w \mid w \in \L(\ub)\ \text{and} \ |w| =
n\}\,.
$$
It is known \cite{Queffelec1987} that the factor complexity of a
fixed point of a primitive substitution is sublinear, i.e., there
exist constants $a,b \in \mathbb{R}$ such that $
\mathcal{C}(n)\leq a n+b$ for all $n \in \N$.

For evaluation of the complexity of an infinite word $\ub$, the
special factors play an important role. Let us denote by ${\rm
Rext}(w) = \{ a\in \A \mid wa \in \mathcal{L}(\ub) \}$  and ${\rm
Lext}(w) = \{ a\in \A \mid aw \in \mathcal{L}(\ub) \}$ the set of
all possible right and left extensions of a factor $w$,
respectively. Clearly $\#{\rm Rext}(w) \geq 1$ for any factor $w$.
If  $\ub$ is recurrent, then also $\#{\rm Lext}(w) \geq 1$. A
factor $w \in \L(\ub)$ is said to be {\em right special} (RS) if $\#{\rm
Rext}(w) \geq 2$ and {\em left special} (LS) if $\#{\rm Lext}(w) \geq
2$. We say that a factor $w$ is {\em bispecial} (BS) if it is at once
right and left special.

\subsection{$\beta$-integers}

In 1957, A. R\'enyi introduced the \emph{$\beta$-expansions} of
positive numbers~\cite{Renyi1957}. Consider a base $\beta > 1$,
then any $x \in [0,\infty)$ can be uniquely expressed in the form
\begin{equation}\label{}
    x = \sum^N_{i = -\infty} x_i \beta^i,
\end{equation}
where $x_i \in \{0,1,\ldots,\lceil\beta\rceil - 1\}$ and
$$
    0 \leq x - \sum^N_{i = n} x_i \beta^i < \beta^n \quad \text{ for all
    } n \leq N+1, n \in \Z.
$$
As it is usual in the everyday used cases of $\beta = 10$ and
$\beta = 2$, we write
$$
    (x)_\beta = x_N x_{N-1} \cdots x_1 x_0 \centerdot x_{-1} x_{-2}
    x_{-3}\cdots
$$
and we call this infinite word the $\beta$-expansion of $x$.

A number $x \in [0,\infty)$ is a \emph{$\beta$-integer} if $x_i =
0$ for all negative indices $i$, i.e., $(x)_\beta = x_N
x_{N-1}\cdots x_1 x_0 \centerdot$. All $\beta$-integers
distributed on the positive real line form a discrete set and the
distances between two neighboring $\beta$-integers are always
$\leq 1$. The set of all these distances can be described
precisely using the \emph{R\'enyi expansion of unity}
$\mathrm{d}_\beta(1) = t_1t_2t_3\cdots$, where $t_1 =
\lfloor\beta\rfloor$ and $0t_2t_3t_4 \cdots$ is the
$\beta$-expansion of $1 - t_1/\beta$. Parry~\cite{Parry1960}
proved that an infinite sequence $t_1t_2t_3\cdots$ of nonnegative
integers is the R\'enyi expansion of unity for some $\beta > 1$ if,
and only if, the following so-called \emph{Parry condition} is
satisfied:
\begin{equation}
    t_it_{i+1}t_{i+2}\cdots \prec t_1 t_2 t_3 \cdots \quad \text{ for all
    } i = 2,3,4,\ldots\ .
\end{equation}
Thurston~\cite{Thurston1989} proved that the distances between
neighboring $\beta$-integers take values in the set $\{\triangle_k
\mid k = 0,1,2, \ldots \}$ with
\begin{equation}\label{eq:distances}
    \triangle_k = \sum_{i=1}^\infty \frac{t_{i+k}}{\beta^i}.
\end{equation}
A number $\beta > 1$ is said to be a \emph{Parry number} if its
set of the distances defined in~\eqref{eq:distances} is finite. In
such cases the distribution of distances between $\beta$-integers
can be coded as an infinite word over a finite alphabet, we denote
this word by $\ubeta$. It is easy to see that $\beta$ is a Parry
number if, and only if, the R\'enyi expansion of unity is eventually
periodic. In particular, we distinguish \emph{simple Parry
numbers} for which
$$
    \mathrm{d}_\beta(1) = t_1t_2\cdots t_m0^\omega
$$
and \emph{non-simple Parry numbers} for which
$$
    \mathrm{d}_\beta(1) = t_1t_2\cdots t_m(t_{m+1} \cdots
    t_{m+p})^\omega.
$$
The positive integers $m,p$ are taken the least possible. This
implies that $t_{m} \neq 0$ in the case of a simple Parry number
and $t_{m} \neq t_{m+p}$ for the non-simple case. As shown by S.
Fabre~\cite{Fabre1995}, the word $\ubeta$ is also the unique fixed
point of the canonical substitution $\vpb$ associated with a Parry
number $\beta$. For a simple Parry number $\beta$, the
substitution $\vpb$ acts on the alphabet $\A=\{0,1,\dots,m-1\}$
and is given by
$$
            \begin{array}{rcl}
            \vpb(0)&=&0^{t_1}1\\
            \vpb(1)&=&0^{t_2}2\\
            &\vdots&\\
            \vpb(m\!-\!2)&=&0^{t_{m-1}}(m\!-\!1)\\
            \vpb(m\!-\!1)&=&0^{t_m}
            \end{array}
$$
For a non-simple Parry number $\beta$, the alphabet is ${\cal
A}=\{0,1,\dots,m+p-1\}$ and
\begin{equation}\label{subst_NS}
            \begin{array}{rcl}
            \vpb(0)&=&0^{t_1}1\\
            \vpb(1)&=&0^{t_2}2\\
            &\vdots&\\
            \vpb(m\!-\!1)&=&0^{t_m}m\\
            \vpb(m)&=&0^{t_{m+1}}(m\!+\!1)\\
            &\vdots&\\
            \vpb(m\!+\!p\!-\!2)&=&0^{t_{m+p-1}}(m\!+\!p\!-\!1)\\
            \vpb(m\!+\!p\!-\!1)&=&0^{t_{m+p}}m.
            \end{array}
\end{equation}
In both cases the substitution is primitive.

As we said, in this paper we focus on the non-simple Parry numbers
$\beta$. For them the incidence matrix of $\vpb$ reads
\begin{equation}\label{eq:icidence_matrix}
    \begin{array}{rcl}
        M = \left(
            \begin{array}{cccc}
              t_1 & 1 & 0 & \cdots  \\
              t_2 & 0 & 1 & \cdots  \\
              \vdots & \vdots & \vdots & \vdots  \\
              t_{m+p-2} & 0 & 0 & \cdots \\
              t_{m+p-1} & 0 & 0 & \cdots
            \end{array}
        \right.
      &
        \begin{array}{c}
          0 \\
          0 \\
          \vdots \\
          0 \\
          1 \\
        \end{array}
      &
        \left.
            \begin{array}{ccc}
              \cdots & 0 & 0 \\
              \cdots & 0 & 0 \\
              \vdots & \vdots & \vdots \\
              \cdots & 0 & 1 \\
              \cdots & 0 & 0 \\
            \end{array}
        \right)
      \\
       & \stackrel{\,\uparrow}{{{\scriptstyle m}\text{\scriptsize-th}}} &  \\
    \end{array}
\end{equation}
Since the substitution $\vpb$ is primitive, the dominant
eigenvalue of $M$ is simple. It is not difficult to prove that
this dominant eigenvalue is equal to $\beta$ and that the vector
$(\triangle_0,\triangle_1,\ldots,\triangle_{m+p-1})^T$, with
$\triangle_i$ defined in~\eqref{eq:distances}, is a right
eigenvector corresponding to it.

For description of the bispecial factors of $\ubeta$ it will be
important to track the last letters of words $\vpb^n(a), a \in \A,
n = 0,1,\cdots$. Therefore, we introduce the following notation.
\begin{dfn}
    For all $k, \ell \in \N$ we define the addition $\oplus : \N \times \N \rightarrow
    \A$ as follows.
    $$
        k \oplus \ell =    \begin{cases}
                            k + \ell & \text{if $k + \ell < m + p$, } \\
                            m + (k + \ell - m \text{ mod } p) & \text{otherwise.}
                        \end{cases}
    $$
    Similarly, if used with parameters $t_i$, we define for all $k, \ell \in \N, k + \ell >
    0$,
    $$
        t_{k \oplus \ell} =    \begin{cases}
                            t_{k + \ell} & \text{if $0 < k + \ell < m + p + 1$,} \\
                            t_{m + 1 + (k + \ell - m - 1 \text{ mod } p)} & \text{otherwise.}
                        \end{cases}
    $$
\end{dfn}
For example, employing this notation one can show that the word
$\vpb^n(a), a \in \A,$ has the suffix $0^{t_{a \oplus n}}(a \oplus
n)$.

\section{Maximal powers and bispecial factors}\label{MaxAndBispecial}

The critical exponent $\E(\ub)$ is defined as the supremum of the
set of indices $\ind{w}$ of all factors $w \in \L(\ub)$. We will
show that the set of factors important for evaluation of $\E(\ub)$
can be significantly reduced.

\begin{lem}\label{lem:BS_and_max_powers}
    Let $w \in \A^*$ have the maximal index in a recurrent
    infinite word $\ub$ between all its conjugates
    and let this index be strictly greater than one. Let $w^\ell
    w'$ be the maximal power of $w$ in $u$, where $\ell \geq 1$ and $w'$ is a proper prefix of $w$. Further, let $b$ be the last letter of $w$
    and let $a$ be the letter such that $w'a$ is a prefix of $w$.
    Then
    \begin{itemize}
        \item[(i)]  $b \notin \Le{w^\ell w'}$ and $a \notin \Rex{w^\ell w'}$,
        \item[(ii)] for $k = 0,1,\ldots,\ell - 1$, $w^k w'$ is a BS factor such that $b \in \Le{w^k w'}$ and $a \in \Rex{w^k
        w'}$.
    \end{itemize}
\end{lem}

\pf

$(i)$ If $b \in \Le{w^\ell w'}$, then the index of a conjugate of
$w$, namely $bwb^{-1}$, is greater than the index of $w$. If $a
\in \Rex{w^\ell w'}$, then $w^\ell w'$ is not the maximal power of
$w$.

$(ii)$ By $(i)$, there exists at least one letter $x \neq b$ such
that $x \in \Le{w^\ell w'}$; hence, $\{b,x\} \subset \Le{w^k w'}$.
Analogously for the case of right extensions.

\pfk

\begin{dfn} \label{dfn:set_B}
    Denote by $\Bub$ the set of (ordered) pairs $(v,w)$ of factors of an infinite word $\ub$
    satisfying the following conditions:
        \begin{itemize}
            \item[(B1)] $v$ is a BS factor,
            \item[(B2)] $wv$ is a power of $w$ in $\ub$.
        \end{itemize}
\end{dfn}
Having this set defined, we can propose the following
straightforward corollary of Lemma~\ref{lem:BS_and_max_powers}.
\begin{cor} \label{cor:critical_exp_alternative}
Given a uniformly recurrent infinite word $\ub$, we have:
    $$
        \E(\ub) = \sup\{\ind{w} \mid (v,w) \in \Bub \text{ for some } v  \}.
    $$
\end{cor}
Of course, the equality is true even if we consider for a given BS
factor $v$ only the shortest $w$ such that $(v,w) \in \Bub$. And
this will be our strategy: we will first find all BS factors $v$
in $\ubeta$ and then the corresponding shortest $w$. Usually, for
a given BS factor $v$, it is not difficult to find $w$ such that
$v$ is a power of $w$ and to verify that $wv$ is a factor of
$\ub$. What may be a problem is to prove that this $w$ is the
shortest such factor. Sometimes it is convenient to use the notion
of return words.
\begin{dfn}
    Let $w \in \L(\ub)$. If $v_L$ and $v_R$ satisfy
    \begin{itemize}
        \item[(i)]  $wv_R = v_Lw \in \L(\ub)$,
        \item[(ii)] there are exactly two occurrences of $w$ in $wv_R =
        v_Lw$,
    \end{itemize}
    then $v_L$ is a \emph{left return word (LRW)} of $w$, $v_R$ is
    a \emph{right return word (RRW)} of $w$, and $wv_L = v_Rw$ is a \emph{complete return word (CRW)} of
    $w$ in $\ub$.
\end{dfn}
For example, if
$$
    \ub =
    00\underline{0}010\underline{0}0100\underline{0}01\underline{0}01000\cdots,
$$
then all LRWs of 0010 visible in this prefix are 0010, 00100, 001.
Thus, a LRW of $w$ may be shorter than $w$ itself!
\begin{lem}\label{lem:RWs_method}
    Let $v$ be a power of $w$ and $\tilde{w}$ a prefix of $v$. If
    $w$ is a LRW of $\tilde{w}$, then $w$ is the root factor of $v$.
\end{lem}
This simple observation turns out to be very useful. In the case
of $\ubeta$, there exists a simple tool for generating all BS
factors. For any BS factor $v$ we will find easily a factor $w$
such that $(v,w) \in \Bubeta$. Then, by a good choice of the
prefix $\tilde{w}$ from the previous lemma, we will prove that
this $w$ is the shortest possible.

\section{Bispecial factors in $\ubeta$}\label{bispecialUNAS}

Throughout the following text, the coefficient $t_1$ from~\eqref{subst_NS} will be greater than $1$.
Corollary \ref{cor:critical_exp_alternative} claims that to get
the critical exponent of $\ubeta$, it suffices to go through all
BS factors $v$ and corresponding $w$ (if it exists) such that
$(v,w) \in \Bubeta$. In what follows, we will take advantage of
having described all BS factors of $\ubeta$ in \cite{Klouda2009}. In
order to present the necessary results we need some more
sophisticated notation for BS factors.
\begin{dfn}
    Let $a,b,c,d \in \A$ such that $a \neq b$ and $c \neq d$. A factor $v \in \A^*$ is an \emph{$(a-c,b-d)$-bispecial factor} of an
    infinite word $\ub$ if both $avc$ and $bvd$ are factors of~$\ub$.
\end{dfn}
In the sequel, the aim is to introduce a mapping (Definition~\ref{map_BS}) which will help us to describe all BS
factors of $\ubeta$ as sequences of words generated by this
mapping from a finite number of short BS factors. We start with
some technical results.
\begin{lem} \label{lem:two_letters_factors}
    Let $a \in \A \setminus \{0\}$. Then the letter
    $$
        b = \max\{ j \mid 0^j \text{ is a suffix of } t_{1}t_2\cdots
        t_{a} \}
    $$
    is a left extension of the factor $a$. Another possible left extension of $a$ is $c$, where
    $$\begin{array}{lcr}
        c = \max\{ j \mid 0^j \text{ is a suffix of } t_{m+1}\cdots
        t_{m+p}\} && \text{for $a = m$}, \\
    c = \max\{ j \mid 0^j \text{ is a suffix of } t_{m+1}\cdots
        t_{m+p}t_{m+1} \cdots  t_{a} \} && \text{for $a>m$}.
        \end{array}
    $$
    There are no
    other left extensions.
\end{lem}
\pf

The statement is a direct consequence of this simple fact: if $t_a
> 0$, then $0$ is a left extension of $a$, if $t_a = 0$ and
$t_{a-1} > 0$, then $1$ is a left extension. Continuing in this
manner we get that $b$ defined as above is always a left extension
of $a$. In fact, $b$ is the last but one letter of $\vpb^a(0)$.

Since $a \geq m$ can appear not only as the last letter of
$\vpb^a(0)$ but also as the last letter of $\vpb^{p+a-m}(m)$, the
letter $c$ can be the other left extension.

\pfk Note that due to the assumption $t_1 > 1$ we must have
$\Le{0} = \A$. If $t_1 = 1$, then clearly $00$ cannot be a factor.

Let us denote throughout the following text
\begin{equation}\label{t_and_z}
t=\min\{t_m, t_{m+p}\} \quad \text{and} \quad \Le{0^tm} = \{0,z\}.
\end{equation}
\begin{cor} \label{cor:values_of _z}
The nonzero left extension $z$ of $0^t m$ is given by
    \begin{equation}\label{eq:def_z}
        z = \begin{cases}
               1 + \max\{ j \mid 0^j \text{ is a suffix of } t_{m+1}\cdots
        t_{m+p}t_{m+1}\cdots t_{m+p-1} \} & t_m > t_{m+p},\\
               1 + \max\{ j \mid 0^j \text{ is a suffix of } t_{1}\cdots t_{m-1} \}     & t_{m+p} > t_m.
            \end{cases}
    \end{equation}
\end{cor}
\pf

Since $t_m = t_{m+p}$ is not admissible, 0 must be a left
extension of $0^tm$. The other left extension $z$ is then given by
the (unique) left extension of $m-1$, if $t = t_m$, or of $m+p-1$,
otherwise.

\pfk

Another consequence of Lemma~\ref{lem:two_letters_factors}
is this:
\begin{cor} \label{cor:beginnings_of_LS_factors}
    If $v$ is a LS factor of $\ubeta$ containing at least one nonzero letter, then one of the following
    factors is a prefix of $v$:
    \begin{itemize}
      \item[(i)] $0^{t_1} 1$,
      \item[(ii)] $0^t m$,
      \item[(iii)] $0^{t_k}k$, if $k > m$ and $t = t_{m+1} = t_{m+2} = \cdots = t_{k-1} = 0$.
    \end{itemize}
\end{cor}
Now, let us introduce the announced mapping that, when iterated, produces all BS factors
from a finite number of some short ones.
\begin{dfn}\label{map_BS}
    Let $\{a,b\}$ be a set of two distinct letters of $\A$. We define:
    $$
        f_L(b,a) = f_L(a,b) = \mathrm{the\ longest\ common\ suffix\ of\ }
        \vpb(a) \mathrm{\ and\ } \vpb(b)
    $$
    and
    $$
        f_R(b,a) = f_R(a,b) = \mathrm{the\ longest\ common\ prefix\ of\ }
        \vpb(a) \mathrm{\ and\ } \vpb(b).
    $$
    If $v$ is an $(a-c,b-d)$-bispecial factor of $\ubeta$, then the
    \emph{$f$-image} of $v$ is the factor
    $$
        f(v) = f_L(a,b)\vpb(v)f_R(c,d).
    $$
\end{dfn}
The $f$-image is defined so that it maps any BS factor to another
one.
\begin{lem}
    Let $v$ be an $(a-c,b-d)$-bispecial factor of $\ubeta$.
    Then we have
    $$
        f_L(a,b) = \begin{cases}
                        0^tm & \text{ if } \{a,b\} = \{m-1, m+p-1\},\\
                        \e & \text{ otherwise},
                    \end{cases}
    $$
    and
    $$
        f_R(c,d) = 0^{\min\{t_{c\oplus 1}, t_{d\oplus 1}\}}.
    $$
    The $f$-image of $v$ is then an $(a'-c',b'-d')$-bispecial factor,
    where $c'$ and $d'$ are the first letters of factors $0^{t_{c\oplus 1} - \min\{t_{c\oplus 1},
    t_{d\oplus 1}\}}(c\oplus 1)$ and $0^{t_{d\oplus 1} - \min\{t_{c\oplus 1}, t_{d\oplus 1}\}}(d \oplus 1)$, respectively, and
    $a'$ and $b'$ are either $0$ and $z$, if $\{a,b\} = \{m-1,m+p-1\}$, or $a \oplus 1$ and $b \oplus
    1$, otherwise.
\end{lem}

Since the $f$-image is again a BS factor, we can construct a
sequence of $f^n$-images of some starting BS factor. It is easy to
see that any $(a-c,b-d)$-bispecial factor containing at
least two nonzero letters has a unique
$f$-preimage, i.e., it is equal to $f(v)$ for a~unique $v$ from Definition~\ref{map_BS}.
This, together with
Corollary~\ref{cor:beginnings_of_LS_factors} (note that $0^{t_k}k$
from $(iii)$ are just $\vpb^{k-m}$-images of $m$), implies that
any BS factor is an $f^n$-image of one of these BS factors:
\begin{itemize}
    \item[(I)] $0^k, 0 < k \leq t_1 - 1$,
    \item[(II)] $0^tm0^\ell, 0 \leq \ell \leq t_1$.
\end{itemize}
In fact, as we shall see, even
$0^tm0^\ell$ is an $f^n$-image of an $(a-c,b-d)$-bispecial factor
$\e$ where $p$ divides $(a-b)$.
\begin{lem} \label{lem:generators_of_BS}
    Let $v$ be an $(a-c,b-d)$-bispecial factor of $\ubeta$. Then
    there exist $n \in \N$ and $a',b',c',d' \in \A$ such that $v$
    is an $f^n$-image of the $(a'-c',b'-d')$-bispecial factor $0^k$ with $0 \leq k \leq t_1 - 1$.
\end{lem}
\pf
The only thing to show is that any
$(0-\tilde{b}, z-\tilde{d})$-bispecial factor $0^tm0^\ell, 0 \leq
\ell \leq t_1$, is an $f^n$-image of $\epsilon$.
It can happen that $0^tm0^\ell$ cannot be BS.
Since $0^tm$ occurs only as a~suffix of $\varphi_\beta^{m+pk}(0), k\geq 0$,
the factor $0^tm$ is always followed in $\ubeta$ either by $\vpb^m(x),
x \in \A$, or by $\vpb^{m+pk}(y), k \geq 1, y \in \A$. Since
$\vpb^{m+pk}(y)$ always begins in $0^{t_1}1$, in order for
$0^tm0^\ell$ can be BS, we need $\vpb^m(x)$ does not begin in
$0^{t_1}1$ for some $x \in \A$. This implies that $t_{x\oplus k}=0$ for $k \in \{1,\dots,m-1\}$. Thus, we have the condition
$\vpb^m(x) = 0^{t_{x \oplus m}}(x \oplus m)$. But in such a case,
$v$ is the $f^{m}$-image of $(a'-x,b'-0)$-bispecial
factor $\e$ with $a' \in \Le{x}$ and $b' = a' + p$.

\pfk
\begin{dfn}
    The $(a-c,b-d)$-bispecial factors $0^k, 0 \leq k < t_1$, will be called \emph{initial}.
\end{dfn}
Thus, all BS factors can be generated from a few short initial
factors applying very simple rule repetitively. This rule can be
even more simplified.
\begin{dfn}\label{dfn:z_n}
Let $n \in \N$ and $n = \ell m  + k, 0 \leq k < m$. Then we put
$$
    z^{(n)} = \begin{cases}
                \e & \text{if $n < m$,} \\
                \vpb^{k}(0^t m) \vpb^{k+m}(0^t m) \cdots \vpb^{(\ell-1)m+k}(0^t
                m) & \text{if $n \geq m$ and $z$ is a multiple of
                $p$,}\\
                \vpb^{n-m}(0^tm) & \text{otherwise.}
            \end{cases}
$$
\end{dfn}
\begin{lem} \label{lem:BS_generator}
Let $v$ be an $(a-c,b-d)$-bispecial factor such that $t_{c \oplus
1} t_{c  \oplus 2}\cdots \preceq t_{d \oplus 1} t_{d \oplus
2}\cdots$. The $f^n$-image of $v$ is equal to $u_1\vpb^n(v)u_2$,
where:
$$
    u_2 = \lcp{\vpb^n(c)}{\vpb^n(d)} = \vpb^n(c)(c \oplus n)^{-1}
$$
and
$$
    u_1 = \begin{cases}
            \e & \text{if $p$ does not divide $a - b$,} \\
            z^{(n + \min\{a,b\})} & \text{otherwise}.
         \end{cases}
$$
\end{lem}

\pf

The fact that $u_2 = \vpb^n(c)(c \oplus n)^{-1}$ is proved in~\cite[Lemma
45,46]{Klouda2009}. As for the form of $u_2$, if $p \nmid a - b$,
then $f_L(a,b) = f_L(a \oplus 1, b \oplus 1) = \cdots = f_L(a
\oplus n, b \oplus n) = \e$, therefore $u_1 = \e$. If $p \mid a -
b$ (assume $a < b$), then $a < m$. We must have $a \oplus (m-a) = b
\oplus (m-a) = m$ and so the $f^{m-a}$-image of $v$ begins in
$z^{(m)} = 0^tm$. The rest is obvious.

\pfk

\section{Main theorem}\label{main}

Having the simple tool for description of all BS factors, it
remains to find the shortest factors which form with them a pair
from $\B(\ubeta)$.

Imagine that we are given a BS factor $v$ which arises from a
nonempty initial factor $0^s, s > 0$. According to
Lemma~\ref{lem:BS_generator}, $v$ can have only one of the
following two forms: \\[1mm]
$(a)$ $v = \vpb^n(0^s)\vpb^n(c)(c \oplus n)^{-1}$ for some $c \in
\A\setminus\{0\}$. In this case it is obvious that $(v,\vpb^n(0))$
is a pair from $\B(\ubeta)$. Moreover, it is not difficult to prove
by induction on $n$ that $\vpb^n(0)$ is the
shortest such factor (see Lemma~\ref{lem:shortest_bez_zn}).\\[1mm]
$(b)$ $v = z^{(n + r)}\vpb^n(0^s)\vpb^n(c)(c \oplus n)^{-1}$ for
some $s > 0$, $c \in \A$, and $r \in \N$. The cases when $r > 0$
will be studied later on. For now assume that $v = z^{(n)}
\vpb^n(0^s) \vpb^n(c)(c \oplus n)^{-1}$. As a direct consequence
of the definitions of $\vpb$ and $z^{(n)}$, we get that $z^{(n)}$ is a
suffix of $\vpb^n(0)$. This yields that
$(v,z^{(n)}\vpb^n(0)(z^{(n)})^{-1})$ is a good candidate for being
an element of $\B(\ubeta)$.

To prove that $w = z^{(n)}\vpb^n(0)(z^{(n)})^{-1}$ is the shortest
possible choice is a bit more problematic than in case $(a)$. In
order to do so we will use Lemma~\ref{lem:RWs_method} with
$\tilde{w} = z^{(n)}$.

Let us take as an example $\beta$ such that $\mathrm{d}_\beta(1) =
33(02)^\omega$. Then $f^2$-image of $(0-1,2-0)$-bispecial factor
$00$ equals to $(0-3,2-0)$-bispecial factor
$$
    v =
    \overbrace{002}^{z^{(2)}}\underbrace{0001000100010002}_{\vpb^2(0)}\underbrace{0001000100010002}_{\vpb^2(0)}\overbrace{000100010001}^{\vpb^2(1)(3)^{-1}}.
$$
Clearly, the factor $w = 0020001000100010$ is the LRW of $002$ and so $w$ is the
root factor of $v$.  It turns out that this argument can be used
in general if we replace $002$ with $z^{(n)}$ and as the factor
$v$ we take
$v^{(n)}$ defined as follows: \\
Denote by $v^{(0)}$ a $(0-c,b-d)$-bispecial factor $0^{s}, s > 0,$
and by $v^{(n)}$ its $f^n$-image. By Lemma~\ref{lem:BS_generator}
we have
\begin{equation}\label{def_of_v_n}
    v^{(n)} = z^{(n)}\vpb^{n}(0^{s}c)(c \oplus n)^{-1}.
\end{equation}
Using these techniques, we will prove the following theorem.
\begin{thm}\label{thm:main_thm}
    Let $z$ be defined as in~\eqref{t_and_z}. If $t_1 \geq 4$ or if $t_1 = 3$ and $p$ does not divide $z$, then
    the critical exponent satisfies
    $$
        \E(\ubeta) = \sup_{n \in \N} \left\{t_1 + \frac{|z^{(n)}| + |\vpb^n(1)| -
        1}{|\vpb^n(0)|}\right\}
    $$
    and the ultimate critical exponent equals
    $$\E^*(\ubeta)=\begin{cases}
                                \beta + \frac{1}{\beta^m -1}(t +
                                \triangle_m) & \text{ if $p$ divides
                                $z$,} \\
                                \beta + \frac{1}{\beta^m}(t + \triangle_m)
                                & \text{ otherwise,}
                            \end{cases}
                            $$
                            where $\Delta_m$ is defined in~\eqref{eq:distances}.
\end{thm}

\section{Proof of the main theorem}\label{proof}

First let us describe all return words of $z^{(n)}$ since they are
playing a crucial role in following sections.

\subsection{Return words of $z^{(n)}$}


We can distinguish three types of return words of $0^tm$. Denote
$X, Y \in \A$ nonzero letters such that $X0^{t_m}m$ and
$Y0^{t_{m+p}}m$ are factors of $\ubeta$; $Y$ is always unique and
$X$ is unique if $m > 1$. Let $v$ be a CRW of $0^tm$, then $v$ is
a suffix of exactly one of the following factors of $\ubeta$:
\begin{itemize}
    \item[(A)] $w_1X0^{t_m}m$,
    \item[(B)] $X0^{t_m}mw_2Y0^{t_{m+p}}m$,
    \item[(C)] $Y0^{t_{m+p}}mw_3Y0^{t_{m+p}}m$,
\end{itemize}
where $w_1$ is long enough and $w_2$ and $w_3$ do not contain
$0^tm$ as a factor. The following lemma is based on this
observation.
\begin{lem} \label{lem:RWs}
    Let $v$ be a CRW of $0^tm$, then $v$ satisfies exactly one of the
    following conditions:
    \begin{itemize}
        \item[(A)] $X0^{t_m}m$ is a suffix of $v$; in this case
        $\vpb^m(0)$ must be also a suffix of $v$,
        \item[(B)] $v$ is a suffix of the $\vpb^m$-image of $0wy$, where
        $y$ is a~letter of the form $sp, s \geq 1,$ and $w \in \A^*$ does not contain $0$ or any multiple of $p$,
        \item[(C)] $v$ is a suffix of $\vpb^p$-image of $mw'm$, where
        $w' \in \A^*$ and $\vpb^p(w')$ does not contain the factor $0^tm$ and $0^tmw'm$ is a~return word of $0^tm$, which is not of type $(A)$.
    \end{itemize}
\end{lem}

\pf

If $X0^{t_m}m$ is a suffix of $v$, then $\vpb^m(0)$ must be a
suffix of $v$ as well since $X0^{t_m}m$ appears in $\ubeta$ only
as a suffix of $\vpb^m(0)$ and $\vpb^m(0)$ contains $0^tm$ only at the end.

Let $v$ be of type $(B)$. Then since $X0^{t_m}m$ occurs only as a~suffix of $\varphi_\beta^m(0)$ and  $Y0^{t_{m+p}}m$ occurs only as a~suffix of $\varphi_\beta^m(sp)$ with $s \geq 1$, the factor $v$ must be a~suffix of $\varphi_\beta^m(0w(sp))$ for some $w \in \A^*$.
Since $v$ has to be a~return word of $0^tm$, the factor $w$ can contain neither $0$ nor any multiple of $p$.

Let $v$ be of type $(C)$. Then since $Y0^{t_{m+p}}m$ occurs only as a~suffix of $\varphi_\beta^p(m)$, the factor $v$ must be a~suffix of $\varphi_\beta^p(mw'm)$.
Since $v$ has to be a~return word of $0^tm$, the factor $\varphi_\beta^p(w')$ cannot contain $0^tm$.

The factor $0^tmw'm$ is obviously a~complete return word of $0^tm$. Assume it is of type $(A)$.
Then $\varphi_\beta^m(0)$ is a~suffix of $w'm$ and so $w'$ contains all letters $\leq m-1$. Consequently,
$\varphi_\beta^p(w')$ contains all letters of $\A$ and also the factor $0^tm$ - a~contradiction.

\pfk
\begin{cor}
    Let $v$ be a CRW of $z^{(n)}, n \geq m$, then $v$ is a suffix of $\vpb^{n-m}$-image of $0v'$ or
    $w'zv'$, where $v'$ is a CRW of $0^tm$ satisfying exactly one of the conditions $(A)$, $(B)$, and $(C)$, and $w' \in \A^*$
    is long enough so that $z^{(n)}$ is a suffix of $\vpb^{n-m}(w'z0^tm)$.
\end{cor}
\pf

If $z$ is not a multiple of $p$, then CRWs of $z^{(n)}, n > m,$
are exactly the $\vpb^{n-m}$-images of the CRWs of $0^tm$.
Otherwise, we need to extend these $\vpb^{n-m}$-images to the left
so that they contain the complete factor $z^{(n)}$. Due to the properties of
$\vpb$, $z^{(n)}$ is always a proper suffix of
$\vpb^{n-m}(00^tm)$. But this might not be true for
$\vpb^{n-m}(z0^tm)$, therefore we must consider a long enough
prolongation by a factor $w'$.

\pfk

\begin{exa}
Let us illustrate the situation where a CRW of $z^{(n)}$ is a~suffix of $\vpb^{n-m}(w'z0^tm)$, however $z^{(n)}$ is not a~suffix of $\vpb^{n-m}(z0^tm)$.
In fact, such a~situation occurs whenever $\vpb^m(z)=0^{t_{m+p}}m$, where $t_{m+p}<t_m$ and $z$ is a~multiple of $p$.
Let $d_\beta(1) = 21^\omega$. The substitution $\vpb$
    then reads: $0 \mapsto 001, 1 \mapsto 01$ and we have $m = 1, p = 1, t = 1,$
    and $0^tm = 01$.
    It is easy to calculate $z^{(3)}=01\vpb(01)\vpb^2(01)$ and $\vpb^{3-m}(z0^tm)=\vpb^2(101)=\vpb(01)\vpb^2(01)$.
\end{exa}
\begin{dfn}
    In the terms of the previous lemma and its corollary, we distinguish three types
    of LRWs, RRWs, and CRWs of $z^{(n)}, n \geq m$: type $(A)$, type $(B)$, and type $(C)$.
\end{dfn}
\begin{exa}
    Let $d_\beta(1) = 221(12)^\omega$. The substitution $\vpb$
    then reads: $0 \mapsto 001, 1 \mapsto 002, 2 \mapsto 03, 3
    \mapsto 04, 4 \mapsto 003$ and we have $m = 3, p = 2, t = 1,$
    and $0^tm = 03$.

    Any CRW of $03$ of type $(A)$ ends in $\vpb^3(0) =
    00100100200100100200100103$, there are three such CRWs:
    \begin{eqnarray*}
       & & 0300100100200100100200100103, \\
       & & 030010010020010010400100100200100100200100103, \\
       & & 030010010020010010020010400100100200100100200100103.
    \end{eqnarray*}

    There are two multiples of $p = 2$ less than $m+p = 5$; we have to consider $y=sp, s = 1,2,$ to get all
    CRWs of type $(B)$. Since $t_{2p} = t_4 > 0$ and $t_p = t_2 > 0$, the factor $w$ from Lemma~\ref{lem:RWs} item $(B)$ is empty.
    Thus $\vpb^m(0p) = \vpb^3(02)$ and $\vpb^m(04) = \vpb^3(04)$ are
    the only sources of CRWs of type $(B)$. We get two CRWs:
    \begin{eqnarray*}
       s = 1 & \rightarrow & 03001001002001003, \\
       s = 2 & \rightarrow & 03001001002001001002001003.
    \end{eqnarray*}

    Having all CRWs of type $(B)$, we can see there are no CRWs of
    type $(C)$ for both $\vpb^2(00100100200100)$ and
    $\vpb^2(00100100200100100200100)$ contain $03$ as a factor.
\end{exa}

\begin{exa}
    Let $d_\beta(1) = 2000(1)^\omega$. The substitution $\vpb$
    then reads: $0 \mapsto 001, 1 \mapsto 2, 2 \mapsto 3, 3
    \mapsto 4, 4 \mapsto 04$ and we have $m = 4, p = 1, t = 0,$
    and $0^tm = 4$.

    There is only one CRW of type $(A)$:
    $40010012001001230010012001001234$.

    As for type $(B)$, there are four multiples of $p = 1$ less than
    $m+p = 5$: $y = sp, s = 1,2,3,4$. But for $s \geq 2$, the only factor
    of the form $0wy$ is $012\cdots y$, i.e., not
    admissible as it contains $p = 1$ between $0$ and $y$. Hence,
    we have only one CRW of type $(B)$: the suffix $404$ of $\vpb^4(01)$.

    For this $\beta$, there exist CRWs of type $(C)$. Take $404$,
    then $w' = 0$ (see Lemma~\ref{lem:RWs} item $(C)$) and $\vpb^p(0) = 001$ does
    not contain $4$. Hence, we get a CRW of type $(C)$: $400104$.
    Now take this factor and again apply $\vpb^p = \vpb$, this
    yields another CRW of type $(C)$: $4001001200104$. Doing the
    same again, we get the third CRW of type $(C)$:
    $4001001200100123001001200104$. And this is the last one since
    the word $\vpb(00100120010012300100120010)$ does contain $4$.
\end{exa}

\subsection{Bispecial factors of type $(I)$}

Now let $u^{(n)}$ be the $f^n$-image of an $(a-c,b-d)$-bispecial
factor $0^s$, where $p$ does not divide $a - b$, $0 < s < t_1$,
and $t_{c \oplus 1} t_{c  \oplus 2}\cdots \preceq t_{d \oplus 1}
t_{d \oplus 2}\cdots$. Then we have by
Lemma~\ref{lem:BS_generator}
$$
    u^{(n)} = \vpb^n(0^s)\vpb^n(c)(c \oplus n)^{-1}.
$$
\begin{lem} \label{lem:shortest_bez_zn}
    The root factor of $u^{(n)}$ is $\vpb^n(0)$.
\end{lem}
\pf


By induction on $n$. The case of $n = 0$ is trivial. For a~greater
$n$, the statement is a direct consequence of the simple fact that
any factor starting in $0^{t_1}1$ and ending in a~nonzero digit has
a~unique $\vpb$-preimage: Assume that $u^{(n)} = \bar{w}^\ell
\bar{w}'$, where $\bar{w}$ is shorter than $\vpb^n(0)$ and $\bar{w}'$ is
a proper prefix of $\bar{w}$. Then $\bar{w}$ must begin in
$0^{t_1}1$ and end in a nonzero letter. Hence, there exists a
unique $\vpb$-preimage of $\bar{w}$ which is shorter than
$\vpb^{n-1}(0)$ and such that $u^{(n-1)}$ is a power of it. A
contradiction.

\pfk We will prove in the sequel that $\vpb^n(0)u^{(n)}$ is a~factor of $\ubeta$. Hence, we will have shown that $(u^{(n)},\vpb^n(0)) \in \B(\ubeta)$.

Now, let us look at the case when $a = 0$ and $b$ is a~multiple of $p$. This assumption means that the $f^m$-image
begins in $0^tm$. Then the $f^n$-image of $(a-c,b-d)$-bispecial
factor $0^s$ reads
$$
    v^{(n)} = z^{(n)}\vpb^n(0^s)\vpb^n(c)(c \oplus n)^{-1}.
$$

For $n < m$, $z^{(n)}$ is empty and hence $v^{(n)}$ equals
$u^{(n)}$ defined above. For $n \geq m$, by Lemma~\ref{lem:RWs_method}, any $w$ such that
$(v^{(n)},w) \in \Bubeta$ must be a LRW of $z^{(n)}$ and since
$z^{(n)}$ is followed by $\vpb^n(0)$, the shortest such $w$ must
be a~LRW of type $(A)$, namely $z^{(n)}\vpb^n(0)(z^{(n)})^{-1}$,
a~conjugate of $\vpb^n(0)$.

Let us summarize what we know so far: if a BS factor, which is an $f^n$-image of some
initial block of zeros, begins in $\vpb^n(0)$, then $\vpb^n(0)$ must be
its root factor, if it begins in $z^{(n)}\vpb^n(0^s)$, the root
factor is $z^{(n)}\vpb^n(0)(z^{(n)})^{-1}$. Since both of these
root factors are of the same length, the greatest index is
attained by the longest one of such BS factors.
Altogether, we have found the root factors of a~significant subset of
all BS factors; these BS factors will be defined as of type $(I)$.
\begin{dfn}
    If $v$ is the $f^n$-image of an $(a-c,b-d)$-bispecial
    factor $0^s$, where $0 < s < t_1$ and either $p$ does not divide $a - b$
    or $a = 0$ and $b$ is a~multiple of $p$, then $v$ is said to be of type $(I)$. If
    $v$ is not of type $(I)$, it is of type $(II)$.
\end{dfn}
The complicated definition of type $(I)$ can be reformulated:
either $p$ does not divide $a-b$ or the $f^m$-image of $0^s$ begins
in $0^tm$. This is not satisfied, e.g., for
$(1-c,(p+1)-d)$-bispecial factors for their $f^{m-1}$-image already
begins in $0^tm$. We will study the BS factors of type $(II)$ in the next
section. In fact, the following holds: the root factors of $(0-c, kp-d)$ BS factors
of type $(I)$ are the LRWs of $z^{(n)}$ of type $(A)$, for a~BS factor
$v$ of type $(II)$ even the LRWs of $z^{(n)}$ of type $(B)$ or $(C)$ can form together with $v$ a~pair from $\Bubeta$.

\begin{dfn}
    We define
    $$
        \E_I(\ubeta) = \sup\{\ind{w}\mid (v,w) \in \Bubeta,\ v \text{of type }(I)\}.
    $$
\end{dfn}

In order to identify the longest BS factors of type $(I)$, we will use some technical results.
We know that the longest common prefix of $\vpb^{n}(c)$ and
$\vpb^{n}(d)$ equals $\vpb^{n}(c)(c \oplus n)^{-1}$, where $t_{c
\oplus 1} t_{c \oplus 2}\cdots \preceq t_{d \oplus 1} t_{d \oplus
2}\cdots$. Due to the Parry condition
$$
t_{c \oplus 1} t_{c  \oplus 2}\cdots \preceq t_{1} t_{2}\cdots
\quad \text{ for all $c \in \A\setminus\{0\},$}
$$
we get that the longest common prefix of $\vpb^{n}(c)$ and
$\vpb^{n}(0)$ equals $\vpb^{n}(c)(c \oplus n)^{-1}$ for all
nonzero $c$. Therefore, to get the longest $u^{(n)}$ and
$v^{(n)}$, the initial factor must be an $(a-c,b-d)$-bispecial factor
$0^{t_1 - 1}$. In order for the condition $(B2)$ from Definition~\ref{dfn:set_B} to be satisfied for
this initial factor and its $f^n$-images, 0 must be one of the
left extension, say $a = 0$. This implies that $0^{t_1}c$ must be
a factor of $\ubeta$; this is always true for $c = 1$. To get the
longest $f^n$-images, the choice $d = 0$ is the best possible and $c$
has to be chosen so that $t_{c \oplus 1} t_{c \oplus 2} \cdots$ is the
greatest possible with respect to the lexicographical order. But since
either $t_c = t_1$ or $t_{m+p} = t_1$ and $c = m$, the
lexicographical maximum is obtained by $c = 1$ due to the Parry
condition. So, the $a,c,$ and $d$ are fixed, $b$ is to be chosen so
that the $f^m$-image of $(0-1,b-0)$ begins in $0^tm$. This always
happens for $b = p$, but any multiple of $p$ will do the same job.

The above explanation implies that if we prove validity of the condition $(B2)$ from Definition~\ref{dfn:set_B},
we will have the right to replace in the definition of $\E_I(\ubeta)$ ``$v$ of type $(I)$'' by ``$v$ is the $f^n$-image of
the $(0-1,p-0)$-bispecial factor $0^{t_1-1}$''.
\begin{lem}
    $z^{(n)}\vpb^n(0^{t_1}1)$ is always a factor of $\ubeta$.
\end{lem}

\pf

For $n < m$ the statement is trivial, let $n \geq m$. We have
already proved that the factor $\vpb^n(0wp)$ contains $z^{(n)}$ as
a suffix, here $w$ is the same $w$ as in Lemma~\ref{lem:RWs}. The
proof then follows from the fact that $0wp0^{t_1}1$ is a factor of
$\ubeta$.

\pfk
\begin{dfn}
    For all $n \in \N$ denote the factor $z^{(n)}\vpb^{n}(0)(z^{(n)})^{-1}$ by
    $w_I^{(n)}$ and the $f^n$-image of the $(0-1,p-0)$-bispecial factor $0^{t_1-1}$ by $v_I^{(n)}$.
\end{dfn}
\begin{pro}\label{pro:result_for_E_I} Let $t_1 \geq 2$. Then
    $$
        \E_I(\ubeta) = \sup\{\ind{w_I^{(n)}} \mid n \in \N \},
    $$
and hence
    $$
        \E_I(\ubeta) = \sup_{n \in \N} \left\{t_1 + \frac{|z^{(n)}| + |\vpb^n(1)| -
        1}{|\vpb^n(0)|}\right\}.
    $$
\end{pro}
\begin{proof}

The proof follows from the simple fact that $w_I^{(n)}$ is a
conjugate of $\vpb^{n}(0)$ and so $|w_I^{(n)}| = |\vpb^{n}(0)|$.

\end{proof}

In~\cite{Klouda2009} we proved that the only $\beta$ for which
$\ubeta$ has an affine factor complexity is the one with $d_\beta(1) =
t_1(0\cdots0(t_1-1))^\omega$. For these special $\beta$s, it is
easy to simplify the formula for $\E_I(\ubeta)$. By induction, one
can prove that $\vpb^n(0)(z^{(n)})^{-1} = 0$, i.e., $|z^{(n)}| =
|\vpb^n(0)| - 1$.
\begin{cor}\label{cor:E_I_sup}
If $d_\beta(1) = t_1(0\cdots0(t_1-1))^\omega$, then
$$
    \E_I(\ubeta) = \sup_{n \in \N} \left\{t_1 + 1 + \frac{|\vpb^{n}(1)| - 2}{|\vpb^n(0)|}\right\}.
$$
\end{cor}

\subsection{Bispecial factors of type $(II)$}

We have split the set of all BS factors of $\ubeta$ into two
subsets: BS factors of type $(I)$ and of type $(II)$. From the
point of view of computing the critical exponent, we are done with
those of type $(I)$. Regarding the type $(II)$, we define an analogue
of $\E_{I}(\ubeta)$:
\begin{dfn}
    $$
        \E_{I\!I}(\ubeta) =  \sup\{ \ind{w} \mid (v,w) \in \B(\ubeta), \text{ $v$ of type
        $(II)$}\}.
    $$
\end{dfn}
In what follows, we will find a~condition under which it holds that
$\E_{I\!I}(\ubeta) < \E_{I}(\ubeta)$ and so $\E(\ubeta) =
\E_{I}(\ubeta)$.

As a first step we will show that any BS factor of type $(II)$ is the
$f^n$-image of the empty word. According to the definition, a BS
factor of type $(II)$ is the $f^n$-image either of the empty word or
of an $(a-c,b-d)$-bispecial factor $0^s, s > 0,$ with $p|(a-b)$ and $a \not = 0$.
For simplicity, assume that $c \neq m$ and $c \geq a$ (the other cases are similar, but more technical),
then $a0^sc = a0^{t_c}c$ is a factor of $\ubeta$. Consequently,
$(a-1)(c-1)$ is a factor as well. If $a-1$ is not zero (and again
for simplicity $c-1 \neq m$), then $(a-2)(c-2)$ is a factor.
Continuing in the same manner, we get $0(c-a)$ is a factor and
the $(a-c,b-d)$-bispecial factor $0^s$ is the $f^{a}$-image of the
$(0-(c-a),b'-d')$-bispecial factor $\e$ (for certain $b',d' \in
\A$). This is the idea of the proof of the following lemma.
\begin{lem}
    For any $(a-c, b-d)$-bispecial factor $v$ of type $(II)$ with $p|(a-b)$,
    there exist $n$ and $b',c',d' \in
    \A$ such that $v$ is the $f^n$-image of $(0-c',b'-d')$-bispecial
    factor $\e$, where, moreover, $b'$ is divisible by $p$.
\end{lem}

Let $v$ be an $f^n$-image of $(0-c,b-d)$-bispecial  factor $\e$,
with $p \mid b$. For $n < m$ we have either $v = \e$ or $v$ is the
$f^\ell$-image of $0^s$ for some $0 \leq \ell < n$ (if there
exists $\ell > 0$ such that $\vpb^{n-\ell}(c) = 0^s(c \oplus (n -
\ell))$). These cases are not interesting since such BS factors
are, in fact, of type $(I)$. The interesting cases are those when
$n \geq m$. For such $n$, the BS factor $v$ begins in $z^{(n)}$
and hence it can only be a power of a~LRW of $z^{(n)}$. But, we have already described
all LRWs of $z^{(n)}$; in particular, for those of type $(A)$ we
know also their maximal power. Thus, given a~BS factor of type
$(II)$, we look for the shortest $LRW$ of $z^{(n)}$ of type $(B)$
or $(C)$ such that the BS factor $v$ is its power.
\begin{lem}\label{lem:proof_of_main_thm}
    Let $w$ be a LRW of $z^{(n)}$ of type $(B)$ or $(C)$. Then:
    \begin{itemize}
        \item[(i)] the index of $w$ in $\ubeta$ is less than $4$,
        \item[(ii)] if, moreover, $p$ does not divide $z$, then
        the index is less than $3$.
    \end{itemize}
\end{lem}
\begin{proof}
    First assume that $n = m$ (and so $z^{(m)} = 0^tm$) and that
    $w$ is of type $(B)$. Then we must have that $w0^tm$ is a suffix
    of $X0^{t_m}mv0^{t_{m+p}}m$ for some $v \in \A^{*}$ not containing $m$
    (see Lemma~\ref{lem:RWs}). This factor is always followed by a RRW of $0^tm$ of type $(C)$
    which is known to be different from $w$ of type $(B)$.
    Therefore, the longest power of $w$ which can appear in
    $\ubeta$ is at most as long as
    $0^{t}mv0^{t_{m+p}}mv0^{t_{m+p}}$. Hence the index of $w$
    is less than 3.

    If $w$ is a LRW of $0^tm$ of type $(C)$, then its index must
    be less than $3$ as well since we know that LRWs of type
    $(C)$ are $\vpb^{p}$-images of those of type $(B)$ (in the sense of
    Lemma~\ref{lem:RWs}).

    Now assume $n > m$, $w$ of type $(B)$, and the case when $p$ does not divide $z$. In
    such a case $z^{(n)} = \vpb^{n-m}(0^tm)$ and so the maximal
    power of $w$ is at most the factor $\vpb^{n-m}(0^{t}mv0^{t_{m+p}}mv0^{t_{m+p}}m)$
    without the last letter, where $v$ is as above. Clearly, the
    index is still less than three. If $w$ is of type $(C)$, we can
    use the same argument as previously and so prove $(ii)$.

    The remaining case is when $p$ divides $z$. Then $z^{(n)}$ is
    not just a $\vpb$-image of $0^tm$ but it is longer. However, if
    $w$ is of type $(B)$, then it is still true that the maximal
    power of it is at most as long as
    $z^{(n)}\vpb^{n-m}(v0^{t_{m+p}}mv0^{t_{m+p}}m)$ without the
    last letter. We know that $z^{(n)}$ is a suffix of
    $z^{(n)}\vpb^{n-m}(v0^{t_{m+p}}m)$ but since it can be possibly
    longer than $\vpb^{n-m}(v0^{t_{m+p}}m)$, the index of $w$ might
    be greater than $3$ (an example of such a situation is
    given below). However, the index cannot get over 4. The case of $w$ of type
    $(C)$ can be again brushed off by the argument used before.
\end{proof}
This result gives us the following upper bound on
$\E_{I\!I}(\ubeta)$.
\begin{cor}\label{cor:proof_of_main_thm}
   The following holds:
    \begin{itemize}
        \item[(i)] $\E_{I\!I}(\ubeta) \leq 4$,
        \item[(ii)] if, moreover, $p$ does not divide $z$, then
        $\E_{I\!I}(\ubeta) \leq 3$.
    \end{itemize}
\end{cor}
As we promised, here is an example of the situation, where a LRW of $z^{(n)}$ of
type $(B)$ or $(C)$ can have the index strictly greater than $3$, in
other words that $\E_{I\!I}(\ubeta) > 3$ can happen.
\begin{exa}
    Let $d_\beta(1) = 22(01)^\omega$. The substitution $\vpb$
    then reads: $0 \mapsto 001, 1 \mapsto 002, 2 \mapsto 3, 3
    \mapsto 02$ and we have $m = 2, p = 2, t = 1, z = 2,$
    and $0^tm = 02$. There is only one CRW of $02$ of type $(B)$,
    namely the suffix of $\vpb^2(02) = \cdots 00202 = \cdots
    0^{t_m}m0^{t_{m+p}}m$. Hence, the factor $v$ from the previous
    proof is the empty word and, since $0020200$ is a factor, the maximal power of
    $02$ reads $02020$. Now consider $z^{(4)} = 02\vpb^{2}(02) =
    0200100100202$. If we compute $\vpb^{2}(0020200)$, we see that
    $$
        02\vpb^{2}(02)\vpb^{2}(02)0010010020 = (02001001002)^3 0
    $$
    is a factor and so the index of the LRW of $z^{(4)}= 02001001002$ is $3 + \frac{1}{11}$.
\end{exa}
Having divided the set of all BS factors into two
disjoint subsets - those of type $(I)$ and of
type $(II)$ - it obviously holds
\begin{equation}\label{}
    \E(\ubeta) = \max\{\E_{I}(\ubeta), \E_{I\!I}(\ubeta)\}.
\end{equation}
Since we know that $\E_{I}(\ubeta) > t_1$,
Theorem~\ref{thm:main_thm} is a simple consequence of Corollary~\ref{cor:proof_of_main_thm} and Proposition~\ref{pro:result_for_E_I}.

\vspace{1cm}
The assumptions of Theorem~\ref{thm:main_thm}
are quite strong and, in fact, we were able to weaken them
slightly. More precisely, we managed to identify the cases when
$\E_{I\!I}(\ubeta) < 2$ and so to prove that in such cases $\E(\ubeta) =
\E_{I}(\ubeta)$ even if $t_1 = 2$. Unfortunately, the proof of
this result is so technical that it becomes almost unreadable. Moreover, there still
remain some $\beta$s such that we are not able to decide whether
$\E(\ubeta)$ is equal to $\E_{I}(\ubeta)$ or to
$\E_{I\!I}(\ubeta)$. In order to show that the latter can happen,
we give the following proposition.
\begin{pro}
    Let the following conditions be satisfied: $\vpb^{m}(p) = m,\ 0p\in\L(\ubeta),\ t_1 = 2,$  and
    $|\vpb^{n}(m)| \geq |\vpb^{n}(1)|$, for $n = 1,2,\ldots,m-1.$  Then
    $$
        \E(\ubeta) =    \E_{I\!I}(\ubeta).
    $$
\end{pro}
\begin{proof}

Obviously, $0p0$ is a factor of $\ubeta$. Hence,
$0^tm0^{t_{m+p}}m0^t = mm$ is a factor as well. Furthermore,
$z^{(n)}\vpb^{n}(p)\vpb^{n}(p)(p \oplus n)^{-1}$ is also a factor
(it follows from the observation that $\vpb^n(a)(a\oplus n)^{-1}$
is a prefix of $\vpb^n(0)$ for every letter $a$) and it is a power
of $z^{(n)}\vpb^{n}(p)(z^{(n)})^{-1}$. This is a conjugate of
$\vpb^n(p)$ and so of the same length. Clearly,
$z^{(n)}\vpb^{n}(p)(z^{(n)})^{-1}$ is a LRW of $z^{(n)}$ of type
$(B)$ and so
$$
    \E_{I\!I}(\ubeta) \geq \sup_{n \geq
    m}\{\ind{z^{(n)}\vpb^{n}(p)(z^{(n)})^{-1}}\}.
$$

We now prove that $\ind{w^{(n)}_{I}} \leq
\ind{z^{(n+m)}\vpb^{n+m}(p)(z^{(n+m)})^{-1}}$ for all $n \in \N$,
i.e., that
$$
2 + \frac{|z^{(n+m)}|  -
        1}{|\vpb^{n+m}(p)|} \geq t_1 + \frac{|z^{(n)}| + |\vpb^n(1)| -
        1}{|\vpb^n(0)|}.
$$
The assumptions and the fact that $z^{(n+m)} = z^{(n)}\vpb^{n}(m)$
yield the inequality
$$
    \frac{|z^{(n)}|  + |\vpb^{n}(m)|- 1}{|\vpb^{n}(m)|} \geq \frac{|z^{(n)}| + |\vpb^n(1)| -
        1}{|\vpb^n(0)|}
$$
equivalent to
$$
    (|z^{(n)}| - 1)(|\vpb^n(0)| - |\vpb^n(m)|) + |\vpb^n(m)|(|\vpb^n(0)| -
    |\vpb^n(1)|) \geq 0.
$$
And this is always true since for $n \geq m$ all members are
nonnegative and for $n < m$ is $|z^{(n)}| = 0$ and
$$
    |\vpb^n(0)| - |\vpb^n(m)| \leq |\vpb^n(0)| - |\vpb^n(1)|.
$$

\end{proof}
\begin{exa}
    Let $d_\beta(1) = 21(1200)^\omega$. The substitution $\vpb$
    then reads: $0 \mapsto 001, 1 \mapsto 02, 2 \mapsto 03, 3
    \mapsto 004, 4 \mapsto 5, 5 \mapsto 2$ and we have $m = 2, p = 4, t = 0, z = 2,$
    and $0^tm = 2$. It holds that $\vpb^m(p) = \vpb^2(4) = m =
    2$ and clearly $0p = 04$ is a factor of $\ubeta$. We will show that
    $\E_{I\!I}(\ubeta)= 3$. Indeed,
    $0^tm\vpb^m(p)(p \oplus m)^{-1} = 2$ is a BS factor with left
    extensions $0$ and $2$ and right extensions $0$ and $2$ as
    well. The BS factors $\vpb^{n-2}(22)(2 \oplus (n-2))^{-1}$ are
    then powers of $\vpb^{n-2}(2)$. We have proved that $\E_{I\!I}(\ubeta) \geq
    3$, the equality holds by Corollary~\ref{cor:proof_of_main_thm}.

    What about $\E_I(\ubeta)$ for this particular $\beta$? After
    some simple computation we get for $n \geq 2$ $w^{(n)}_{I} = \vpb^{n-2}(20010010)$
    and $v^{(n)}_{I} = \vpb^{n-2}(20010010200103)(3\oplus
    (n-2))^{-1}$. Hence
    $$
        \E_I = \sup_{n \geq 2}\left\{ 3 - \frac{|\vpb^{n-2}(010)| - |\vpb^{n-2}(3)| +
        1}{|\vpb^{n-2}(20010010)|}\right\} < 3.
    $$
\end{exa}

\section{The ultimate critical exponent}

In this section, we will find the ultimate critical exponent of
$\ubeta$ under the assumptions of Theorem~\ref{thm:main_thm}.
Using the formula for $\E(\ubeta)$ -- proven in the previous
section -- our task is to calculate the following limit
$$\E^*( \ubeta) =\lim_{n \to \infty}\left(t_1+\tfrac{|z^{(n)}|+|\vpb^n(1)|-1}{|\vpb^n(0)|}\right).$$

\subsection{Auxiliary limits}

In order to be able to compute the desired limit, we will need
some technical results. For calculation of the lengths of $z^{(n)}$,
$\vpb^n(0)$,  and  $\vpb^n(1)$, we will use the notions of the Parikh
vector $\Psi(w)$ and the incident matrix $M$ of a primitive
substitution $\varphi$. Recall that $\vec{e}$ stands for a column
vector whose all entries are equal to one.  As a~simple consequence of~\eqref{parikh}, we get
the following lemma.
\begin{lem}\label{lem:number_of_letters_in_vp_to_n}
    For all $n \in \N$ and $w \in \A^*$ we have
    $$
        |\varphi^n(w)| = \Psi(w)M^n\vec{e}\,.
    $$
\end{lem}
Since the matrix $M$ is primitive, there exists a simple dominant
eigenvalue $\beta \in \R$ such that any other eigenvalue is in
modulus less than $\beta$. Denote $\vec{x}$ and $\vec{y}$ a left
and right eigenvector for $\beta$ (they can be chosen nonnegative), i.e.,
$$
    \vec{x}M = \beta \vec{x} \quad
    \text{and}\quad M\vec{y} = \beta
    \vec{y}.
$$
Let $\mathcal{J}$ be the Jordan canonical form of $M$ such that
\begin{equation}\label{eq:jordan_form}
    M = P\mathcal{J}P^{-1} = P\begin{pmatrix}
      \beta & \vec{0} \\
      \vec{0}^{\ T} & \mathcal{J}_{22} \\
    \end{pmatrix}P^{-1},
\end{equation}
where $\vec{0} = (0,\ldots,0)$ is a zero vector of the
corresponding size and $\mathcal{J}_{22}$ contains the
Jordan blocks corresponding to the eigenvalues different from
$\beta$. With this notation, we see that $P$ can be chosen so
that the first column of $P$ is $\vec{y}$ and the first row of
$P^{-1}$ is $\vec{x}$, but with the condition that for the eigenvectors
in question we have $\vec{x}\vec{y} = 1$. Moreover, for any Jordan block in
$\mathbb{R}^{d\times d}$ and
 an exponent $n \in \mathbb{N}$, one can prove by induction
\begin{equation} \label{eq:power_of_jordan_block}
{\begin{pmatrix}
      \lambda & 1 & \cdots & 0 \\
      0 & \lambda & \cdots & 0 \\
      \vdots & \vdots &  & \vdots \\
      0 & 0 & \cdots & \lambda \\
    \end{pmatrix}}^{\!\!\!n} = \begin{pmatrix}
      \lambda^n & \lambda^{n-1}\binom{n}{1} & \lambda^{n-2}\binom{n}{2} & \cdots & \lambda^{n-d+1}\binom{n}{d-1} \\
      0 & \lambda^n & \lambda^{n-1}\binom{n}{1} & \cdots & \lambda^{n-d+2}\binom{n}{d-2} \\
      \vdots & \vdots & \vdots &  & \vdots \\
      0 & 0 & 0 & \cdots & \lambda^n \\
    \end{pmatrix}
\end{equation}
All these facts allow us to prove easily the following lemma.
\begin{lem} \label{lem:limit_the_case_of_letters}
    Let $M$ be a primitive nonnegative matrix $M$ with the dominant eigenvalue $\beta$ and  $P$ the matrix
    defined by \eqref{eq:jordan_form}. Then
    \begin{itemize}
 \item[(i)]
    $$
        \lim_{n \to \infty} \frac{1}{\beta^n}M^n = P\begin{pmatrix}
      1 & \vec{0} \\
      \vec{0}^{\ T} & \Theta \\
    \end{pmatrix}P^{-1},
    $$

      \item[(ii)]
    $$
        \lim_{n \to \infty} \frac{1}{\beta^{s+nr}}\sum_{i=0}^{n-1}M^{s+ri} = \frac{1}{\beta^r - 1}P
        \begin{pmatrix}
      1 & \vec{0} \\
    \vec{0}^{\ T} & \Theta \\
      \end{pmatrix}P^{-1},
    $$
    \end{itemize}
    where $\Theta$ is a zero matrix of the corresponding size and  $s,r \in \mathbb{N}$, $r$ positive.
\end{lem}

\begin{proof}
    To compute the second limit, we consider the Jordan form~\eqref{eq:jordan_form} and~\eqref{eq:power_of_jordan_block}.  We get
    $$
        \lim_{n \to \infty} \frac{1}{\beta^{s+nr}}
        \sum_{i=0}^{n-1}\mathcal{J}^{s+ri}_{22} = \Theta.
    $$
    The proof then follows by the simple fact that
    $$
        \frac{\sum_{i=0}^{n-1}\beta^{s+ri}}{\beta^{s+rn}} \longrightarrow \frac{1}{\beta^r - 1} \quad \text{as } n \to \infty.
    $$
The value of the first limit is obvious.
\end{proof}
In both cases, we got a~very similar expression on the right-hand
side. It can be even more simplified.
\begin{lem}\label{lem:the_constant_C_M}
    Let $M$ be a primitive matrix and let $\vec{y}$ be a right
    eigenvector corresponding to the dominant eigenvalue $\beta$.
    Then
    $$
        P
        \begin{pmatrix}
      1 & \vec{0}^{\ T} \\
      \vec{0} & \Theta \\
      \end{pmatrix}P^{-1}
\vec{e} = C_M \vec{y},
    $$
    where $P$ is given by~\eqref{eq:jordan_form} and $C_M$ is
    a~positive constant depending on the choice of $\vec{y}$.
\end{lem}
\begin{proof}
    As we said before, the first row of $P^{-1}$ is the left
    eigenvector of the dominant eigenvalue such that $\vec{x}\vec{y} =
    1$. Hence, we get
    $$
        \begin{pmatrix}
      1 & \vec{0}^{\ T} \\
      \vec{0} & \Theta \\
      \end{pmatrix}P^{-1}
         \vec{e} =
        \begin{pmatrix}
      \vec{x} \\
      \Theta' \\
      \end{pmatrix}
       \vec{e} =    \vec{x}   \vec{e}
        \begin{pmatrix}
      1  \\
      \vec{0}  \\
      \end{pmatrix}.
    $$
    This, along with that the first column of $P$ is $\vec{y}$,
    conclude the proof. Moreover, we get $C_M =  \vec{x}   \vec{e}$.
\end{proof}

\subsection{The case of $\ubeta$}

The incidence matrix $M$ of the non-simple Parry substitution $\vpb$,
having $\ubeta$ as its fixed point, is defined in~\eqref{eq:icidence_matrix}.
As we have already mentioned, the components of its right eigenvector $ \vec{y}_\beta  $ corresponding to the
eigenvalue $\beta$ represent distances between the consecutive
$\beta$-integers (see \eqref{eq:distances}), i.e.,
$$\vec{y}_\beta = (\triangle_0,
      \triangle_1 ,
      \triangle_2, \dots , \triangle_{m+p-1})^T.$$
We have now at our disposal all we need to compute the limit equal to $E^*(\ubeta)$.
In order to simplify the notation, let us omit the index $\beta$ in $\vpb$ and $\vec{y}_\beta$.
Let us start with calculation of the relevant limits.

\begin{lem}\
    \begin{itemize}
        \item[(i)]
            $$
                \lim_{n \to \infty}
                \frac{|\varphi^n(1)|}{|\varphi^n(0)|} = \triangle_1 =
                \beta - t_1,
            $$
        \item[(ii)]
            $$
                \lim_{n \to \infty}
                \frac{|z^{(n)}|}{|\varphi^n(0)|} =
                            \begin{cases}
                                \frac{1}{\beta^m -1}(t +
                                \triangle_m) & \text{ if $p$ divides
                                $z$,} \\
                                \frac{1}{\beta^m}(t + \triangle_m)
                                & \text{ otherwise.}
                            \end{cases}
            $$
    \end{itemize}
\end{lem}
\begin{proof}
    By Lemma~\ref{lem:number_of_letters_in_vp_to_n} we have
    $$
        \frac{|\varphi^n(1)|}{|\varphi^n(0)|} =
        \frac{\frac{1}{\beta^n}\Psi(1)M^n\vec{e}}
        {\frac{1}{\beta^n}\Psi(0)M^n\vec{e}}.
    $$
    Due to Lemmas~\ref{lem:limit_the_case_of_letters}
    and~\ref{lem:the_constant_C_M}, this tends to
    $$
        \frac{C_{M}\Psi(1)\vec{y}}{C_{M}\Psi(0)\vec{y}}
        =  \triangle_1 = \beta - t_1
    $$
    as $n$ goes to infinity.

    We divide the proof of~$(ii)$ into two parts. First assume
    that $p$ does not divide $z$. Then, using the same techniques as
    for~$(i)$, we get
    \begin{multline*}
    \lim_{n \to \infty} \frac{|z^{(n)}|}{|\varphi^n(0)|} =
    \lim_{n \to \infty} \frac{t|\varphi^{n-m}(0)| +
    |\varphi^{n-m}(m)|}{|\varphi^n(0)|} = \\ = \frac{1}{\beta^m}\frac{\left(t\Psi(0) +
    \Psi(m)\right)\vec{y}}{\Psi(0)\vec{y}}=
    \frac{1}{\beta^m}(t + \triangle_m).
    \end{multline*}

    Second, let $p$ divide $z$ and let  $0 \leq k <
    m$. Then by Definition~\ref{dfn:z_n} we have
    $$
        |z^{(n m + k)}| = t\sum_{j = 0}^{n - 1}|\varphi^{k + j m}(0)| + \sum_{j = 0}^{n - 1}|\varphi^{k + j
        m}(m)|.
    $$
    Now, by Lemma~\ref{lem:limit_the_case_of_letters} we get
    \begin{multline*}
    \lim_{n \to \infty} \frac{|z^{(n m + k)}|}{|\varphi^{n m + k}(0)|} =
    \lim_{n \to \infty}\frac{\left(t\Psi(0) +
    \Psi(m)\right) \frac{1}{\beta^{n m + k}}\sum_{i=0}^{n-1}M^{k+mi} \vec e
    }{\Psi(0) \frac{1}{\beta^{n m + k}}M^{mn+k} \vec e
    }=\\\frac{1}{\beta^m - 1}\frac{\left(t\Psi(0) +
    \Psi(m)\right)\vec{y}}{\Psi(0)\vec{y}}=
    \frac{1}{\beta^m - 1}(t + \triangle_m).
    \end{multline*}
    Since the resulting expression does not depend on $k$, the
    proof is finished.
\end{proof}

Using the obtained limit values, we get the statement of Theorem~\ref{thm:main_thm} concerning the ultimate critical exponent:
 $$\E^*(\ubeta)=\begin{cases}
                                \beta + \frac{1}{\beta^m -1}(t +
                                \triangle_m) & \text{ if $p$ divides
                                $z$,} \\
                                \beta + \frac{1}{\beta^m}(t + \triangle_m)
                                & \text{ otherwise.}
                            \end{cases}
                            $$

\section{Comments}
\begin{itemize}

\item
Our method for calculation of the critical exponent and the ultimate critical exponent
can be used for any primitive substitution $\varphi$. It can be shown for any such substitution
that all BS factors arise when applying the map $f$ from Definition~\ref{map_BS} repeatedly on a~finite
number of initial BS factors.

\item The sequences of BS factors we have studied are, in terms of
Lemma~\ref{lem:BS_and_max_powers}, maximal powers of some factor
$w$ minus the prefix $w$. Krieger in~\cite{Krieger2007} considered
directly sequences of maximal powers called $\pi$-sequences. There
is a strong relation between these two sequences, it holds
(omitting some technicalities) that $(v_i)_{i \geq 0}$ is a
sequence of BS factors which are powers of $w_i$ and $(v_i,w_i)
\in \B(\ub)$ if and only if $(w_iv_i)_{i \geq 0}$ is a
$\pi$-sequence. However, she studied the general case where she
needed ``only'' to know that there are only a finite number of
these $\pi$-sequences. The method used in this paper is a feasible
way of how to identify all the $\pi$-sequences for a particular
substitution.

\item In~\cite{Balkova2009} the values of the critical exponent
for quadratic non-simple Parry numbers are studied. In this
special case, the R\'enyi expansion of unity
$d_\beta(1)=t_1t_2^{\omega}$, hence the period length $p=1$ and
$p$ divides then $z$ automatically. In this case, we are able to
decide when $E^*(\ubeta)=E(\ubeta)$~\cite[Theorem
5.3]{Balkova2009}.

\item The exact value of the constant  $C_{M}$ from Lemma
\ref{lem:the_constant_C_M} was not necessary for calculation of
our limits. However, its value is computed in~\cite{Balkova2008b}
in case of canonical substitutions associated with simple and
non-simple Parry numbers:
$$\begin{array}{rl}
    C_{M} = \frac{\beta - 1}{\beta^m(\beta^p - 1)}P'(\beta) & \text{for non-simple Parry numbers},\\
    C_{M} = \frac{\beta - 1}{(\beta^m - 1)}P'(\beta) & \text{for simple Parry numbers},

\end{array}
$$
where $P(x)$ is the Parry polynomial of $\beta$ defined ibidem.
\item An essential part of this paper is devoted to BS factors.
This notion plays an important role in the study of many
characteristics of infinite words, e.g., factor complexity,
palindromic complexity, return words, abelian complexity etc.
\end{itemize}

\section*{Acknowledgement}

We acknowledge financial support by the Czech Science Foundation
grant 201/09/0584 and by the grants MSM6840770039 and LC06002 of
the Ministry of Education, Youth, and Sports of the Czech
Republic. We also thank the CTU student grant
SGS10/085/OHK4/1T/14.

    \bibliographystyle{plain}
    \bibliography{publications}

\begin{thebibliography}{10}

\bibitem{Balkova2008b}
L{\!'}. Balkov\'a, J.-P. Gazeau, and E~Pelantov\'a.
\newblock Asymptotic behavior of beta-integers.
\newblock {\em Lett. Math. Phys.}, 84:179--198, 2008.

\bibitem{Balkova2009}
L.~Balkov\'{a}, K.~Klouda, and E.~Pelantov\'{a}.
\newblock Repetitions in beta-integer.
\newblock {\em Lett. Math. Phys.}, 87:181--195, 2009.

\bibitem{Berthe2006}
V.~Berth\'e, Ch. Holton, and L.~Q. Zamboni.
\newblock Initial powers of sturmian sequences.
\newblock {\em Acta Arith.}, 122:315--347, 2006.

\bibitem{Burdik1998}
C.~Burd\'ik, Ch. Frougny, J.~P. Gazeau, and R.~Krejcar.
\newblock Beta-integers as natural counting systems for quasicrystals.
\newblock {\em J. Phys A, Math. Gen.}, 31:6449--6472, 1998.

\bibitem{Carpi2000}
A.~Carpi and A.~de~Luca.
\newblock Special factors, periodicity, and an application to sturmian words.
\newblock {\em Acta Inf.}, 36(12):983--1006, 2000.

\bibitem{Damanik2000}
D.~Damanik.
\newblock Singular continuous spectrum for a class of substitution hamiltonians
  {II}.
\newblock {\em Lett. Math. Phys.}, 54:25--31, 2000.

\bibitem{Damanik2002}
D.~Damanik and D.~Lenz.
\newblock The index of sturmian sequences.
\newblock {\em Eur. J. Comb.}, 23(1):23--29, 2002.

\bibitem{Fabre1995}
S.~Fabre.
\newblock Substitutions et beta-syst\`emes de num\'{e}ration.
\newblock {\em Theoret. Comput. Sci.}, 137:219--236, 1995.

\bibitem{Klouda2009}
K.~Klouda and E.~Pelantov\'{a}.
\newblock Factor complexity of words associated with non-simple parry numbers.
\newblock {\em Integers}, 9(3):281–309, 2009.

\bibitem{Krieger2007}
D.~Krieger.
\newblock On critical exponents in fixed points of non-erasing morphisms.
\newblock {\em Theor. Comput. Sci.}, 376(1-2):70--88, 2007.

\bibitem{Mignosi1992}
F.~Mignosi and G.~Pirillo.
\newblock Repetitions in the fibonacci infinite word.
\newblock {\em RAIRO Inform. Theor. Appl.}, 26:199--204, 1992.

\bibitem{Parry1960}
W.~Parry.
\newblock On the $\beta$-expansions of real numbers.
\newblock {\em Acta Math. Acad. Sci. Hunger.}, 11:401--416, 1960.

\bibitem{Queffelec1987}
M.~Queff\'{e}lec.
\newblock {\em Substitution dynamical systems--spectral analysis}, volume 1284
  of {\em Lecture Notes in Mathematics}.
\newblock Springer-Verlag, Berlin, 1987.

\bibitem{Renyi1957}
A.~R\'{e}nyi.
\newblock Representations for real numbers and their ergodic properties.
\newblock {\em Acta Math. Acad. Sci. Hungar.}, 8:477--493, 1957.

\bibitem{Thue1912}
A.~Thue.
\newblock \"{U}ber die gegenseitige {L}oge gleicher {T}eile gewisser
  {Z}eichenreihen.
\newblock {\em Norske Vid. Skrifter I Mat.-Nat. Kl. Chris.}, 8:1--67, 1912.

\bibitem{Thurston1989}
W.~Thurston.
\newblock Groups, tilings and finite state automata.
\newblock AMS Colloquium Lecture Notes, 1989.

\end{thebibliography}
\end{document}